\documentclass[11pt]{amsart}
\usepackage[colorlinks,linkcolor=blue, anchorcolor=blue, linktocpage=true, urlcolor=blue, citecolor=blue]{hyperref}
\usepackage{amssymb}
\usepackage{mathrsfs}
\usepackage{enumerate}
\usepackage{hyperref}
\usepackage{color}
\usepackage{graphicx}
\usepackage{txfonts}

\newtheorem{theorem}{Theorem}[section]
\newtheorem{definition}{Definition}[section]
\newtheorem{proposition}[theorem]{Proposition}
\newtheorem{lemma}[theorem]{Lemma}
\newtheorem{corollary}[theorem]{Corollary}
\newtheorem{claim}[theorem]{Claim}
\newtheorem{remark}[theorem]{Remark}

\newtheorem{theoremalph}{Theorem}
\newtheorem{corollaryalph}{Corollary}

 \def\NN{{\mathbb N}} 

 \def\RR{{\mathbb R}}

\def \F{\EuScript{F}}

\def \M{\mathcal{M}}
\def \N{\mathcal{N}}

\def \T{\mathcal{T}}

\def\N {\Bbb N}
\def\M {{\mathcal M}}
\def\F {{\mathcal F}}

\def\P{\mathcal P}

\def \E{\mathscr{E}}
\def \F{\mathscr{F}}
\def \N{\mathscr{N}}

\begin{document}

\newpage
\title[SRB measures for partially hyperbolic flows]
{SRB measures for partially hyperbolic flows with mostly expanding center}	
\author{Zeya Mi}
	
\address{School of Mathematics and Statistics,
Nanjing University of Information Science and Technology, Nanjing 210044, Jiangsu, P.R. China.}

\email{\href{mailto:mizeya@nuist.edu.cn}{mizeya@nuist.edu.cn}, \href{mailto:mizeya@163.com}{mizeya@163.com}}

\author[Biao You]{Biao You}

\address{Department of Mathematics, Soochow University, Suzhou 215006, Jiangsu, P.R. China.}

\email{\href{mailto:sudayoubiao@163.com}{sudayoubiao@163.com}}

\author{Yuntao Zang}
\address{Department of Mathematics, Soochow University, Suzhou 215006, Jiangsu, P.R. China.}
\address{Laboratoire de Math\'ematiques d'Orsay, CNRS - Universit\'e Paris-Sud, Orsay 91405, France}
\email{\href{mailto:yuntaozang@suda.edu.cn}{yuntaozang@suda.edu.cn}}


\thanks{Z.Mi\ was partially supported by NSFC 11801278}
\thanks{}

\date{\today}

\maketitle

\begin{abstract}
We prove that a partially hyperbolic attractor for a $C^1$ vector field with two dimensional center supports an SRB measure. In addition, we show that if the vector field is $C^2$, and the center bundle admits the \emph{sectional expanding} condition w.r.t. any Gibbs $u$-state, then the attractor can only support finitely many SRB/physical measures whose basins
cover Lebesgue almost all points of the topological basin. 
The proof of these results has to deal with the difficulties which do not occur in the case of diffeomorphisms.
\end{abstract}

\maketitle

\section{Introduction and main results}
In the 1960s and 1970s, Sinai, Ruelle and Bowen \cite{Sin72,BoD75, Rue76,Bow77} systematically studied the chaotic behaviors on uniform hyperbolic systems by using statistical mechanics and smooth ergodic theory. The central idea is the construction of a natural invariant measure called SRB measures(see \cite{Y02}). The existence and finiteness of SRB measures for uniformly hyperbolic systems are well understood. Researchers then focus on more general settings like partially hyperbolic systems and other non-uniformly hyperbolic systems. There are various non-uniform conditions that are proposed to get the existence and finiteness of SRB measures (see e.g. \cite{AlB00,BoV00,APP,CaY16, ADL17,LY17,MCY17,AnC18,A19}) both on diffeomorphisms and flows.
In this paper, we are interested in the problem on the existence and finiteness of the SRB measures for partially hyperbolic flows when the center bundle has low dimension or admits some hyperbolicity.

Let $\phi$ be the flow generated by a $C^1$ vector field $X$ on a compact Riemannian manifold $M$. A Borel probability measure $\mu$ is $\phi$-invariant if $\mu(\phi_t(A))=\mu(A)$ for every measurable subset $A$ and $t\in \RR$. Roughly speaking, we say a probability measure is an SRB measure if it is a $\phi$-invariant measure with entropy being equal to the sum of its positive Lyapunov exponents (see Definition \ref{SRB}). 

We mainly study the existence and finiteness of SRB measures supported on \emph{attractors}.
A compact subset $\Lambda\subset M$ is an attractor if there exists an open neighborhood $U$ of $\Lambda$ such that
$\phi_{t}(U)\subset \overline{U}$ for every $t\ge 0$ and $\Lambda=\bigcap_{t\ge0}\phi_{t}(U).$

\begin{theoremalph}\label{TheA}
Let $X$ be a $C^{1}$ vector field on a compact manifold $M$, and let  $\Lambda$ be an attractor with a partially hyperbolic splitting $T_{\Lambda}M=E^{ss}\oplus E^{c}\oplus E^{uu}$. If $\dim E^{c}=2,$ then there is an SRB measure supported on $\Lambda$.
\end{theoremalph}

For partially hyperbolic systems with one dimensional center, it was showed by
Cowieson-Young \cite{Cow05} and Crovisier-Yang-Zhang \cite{CDJ20} that any attractor admits an SRB measure. 
Theorem \ref{TheA} suggests that the condition of one dimensional center bundle for diffeomorphisms can be relaxed to two dimensional center bundle for flows. Different to \cite{Cow05}, where they use the tool of random perturbations, we mainly borrow some ideas from \cite{CDJ20} on the argument of dominated entropy formulas.

Let $\phi$ be the flow generated by a $C^1$ vector field $X$. A $\phi$-invariant probability measure $\mu$ is called a \emph{physical measure} if its basin
	$$ 
	B(\mu)=\Big\{   x\in M:\ \lim\limits_{T\to \infty}\frac{1}{T}\int_{0}^{T}\delta_{\phi_{t}(x)}dt=\mu \Big\}
	$$ 
has positive Lebesgue measure. SRB measures are sometimes expected to be physical. 
A well-known fact is that any ergodic hyperbolic SRB measure is a physical measure when the vector field $X$ is $C^2$. 

\medskip
We will also study the finiteness of SRB/physical measures for partially hyperbolic flows, for this we will add some hyperbolicity of the center bundle.


Let $\phi_{1}$ be the time one map of the flow $\phi$. Assume that $\Lambda$ is a partially hyperbolic attractor exhibiting strong unstable direction $E^{uu}$. Let $h_{\mu}(\phi_1,\mathcal{F}^{u})$ be the entropy along unstable foliation. For more detail, see Definition \ref{GS}. A good candidate of SRB measures in many systems is called \emph{Gibbs $u$-state}. 
\begin{definition}
	An invariant probability measure $\mu$ is called a Gibbs $u$-state if $$h_{\mu}(\phi_1,\mathcal{F}^{u})= \int \log|\det D\phi_{1}|_{E^{uu}}|d\mu.$$
\end{definition}

Let us remark that this kind of definition of Gibbs $u$-state was proposed by previous works \cite{CDJ20, MC20}. By \cite{L84}, when the system is $C^2$, $\mu$ is a Gibbs $u$-state iff the conditional measures of $\mu$ along strong unstable manifolds are absolutely continuous w.r.t. the Lebesgue measures. The latter property is the original definition of Gibbs $u$-states introduced by \cite{PeS82}.

We next propose a non-uniform condition on Gibbs $u$-states.

\begin{definition}\label{Gibbs expanding}
Let $E$ be an invariant sub-bundle of $T_{\Lambda}M$. We say that $E$ is \emph{Gibbs sectional expanding} if ${\rm dim}E^c>1$ and for every Gibbs $u$-state $\mu$ supported on $\Lambda$, we have
$$
\lim_{t\to +\infty}\frac{1}{t}\log \left|{\rm det} D \phi_t|_{L_x}\right|>0
$$
for every two-dimensional subspace $L_x\subset E_x$ and $\mu$-almost every $x$.
\end{definition}

Note that the above limit exists by Oseledets theorem \cite[Theorem S.2.9]{Kat95}.
Under the condition of Gibbs sectional expanding on the center bundle, we can show the finiteness of SRB/physical measures.

\begin{theoremalph}\label{TheB}
Let $\phi$ be a $C^{2}$ flow on a compact manifold $M$ and let  $\Lambda$ be an attractor with a partially hyperbolic splitting $T_{\Lambda}M=E^{ss}\oplus E^{c}\oplus E^{uu}$. Assume that $E^{c}$ is Gibbs sectional expanding. Then there are finitely many physical measures (which are SRB) supported on $\Lambda$ and their basins cover a full Lebesgue measure subset of the topological basin $B(\Lambda)$\footnote{$
B(\Lambda)=\{x: \omega(x)\subset \Lambda\},$ where $\omega(x)$ is $\omega$-limit set of $x$.} of $\Lambda$.
%
\end{theoremalph}

When $\phi$ is transitive, then we get the uniqueness of SRB (physical) measures.

\begin{corollaryalph}\label{uniqueness}
Under the assumption of Theorem \ref{TheB}, if we assume that $\phi$ is transitive on $\Lambda$, then there exists a unique SRB measure (which is physical) supported on $\Lambda$, whose basin covers a full Lebesgue measure subset of the topological basin $B(\Lambda)$.
\end{corollaryalph}

The Gibbs sectional expanding condition may be seen as a local version of the usual sectional expanding condition in the definition of singular hyperbolic attractor(see \cite{MM08,Z08}). The existence and finiteness of physical measures for singular hyperbolic attractors has been obtained earlier in \cite{APP,LY17,A19} under different settings. We emphasize that for singular hyperbolic attractors, the main difficulty is the possibility of existence of singularities, which is different to ours(see Lemma \ref{System is not a singularity}).

Indeed, the Gibbs sectional expanding condition is much similar to the mostly expanding condition on partially hyperbolic diffeomorphisms, where it requires that every Gibbs $u$-state admits only positive central Lyapunov exponents. For mostly expanding diffeomorphism, in the previous works \cite{AnC18, MCY17}, the authors have proved the finiteness of physical measures.

The key idea of establishing the finiteness of physical measures in Theorem \ref{TheB} is to show the existence of `typical Pesin unstable manifolds' with uniform (lower) size in which Lebesgue almost every point belongs to the basin of the corresponding ergodic physical measures. This is done for diffeomorphisms considered in \cite{AnC18, MCY17} by iterating some special disks tangent to the center unstable direction. Due to the involvement of flow direction, we cannot achieve this goal just by studying the time one map of the flow. 
In contrast to the previous strategy, we use the argument on periodic approximation and some uniform estimate on central direction to find hyperbolic periodic orbit $\gamma_{\alpha}$ w.r.t. each (ergodic) physical measure $\mu_{\alpha}$ such that 
\begin{itemize}
\item each $\gamma_{\alpha}$ is homoclinically related $\mu_{\alpha}$-typical points;
\smallskip
\item the unstable manifolds of these periodic orbits admit uniform size.
\end{itemize} 
This implies the existence of typical unstable manifolds with uniform size. Let us remark that the philosophy of using hyperbolic periodic orbit to analyze physical measures appears in recent works for the case of diffeomorphism (see e.g. \cite{HCMP,dvy16,yjg19,MC20}).

\bigskip
{\bf{Acknowledgements.}} We would like to express our gratitude to Prof. Dawei Yang for his instructive and useful suggestions. We also thank Jinhua Zhang and Rui Zou for many helpful discussions.

\section{Preliminary}
Throughout this section, $X$ denotes a $C^1$ vector field on a compact Riemannian manifold $M$. Let $\phi$ be the flow generated by $X$. A point $\sigma\in M$ is called a \emph{singularity} of $X$ if $X(\sigma)=0$. Denote by ${\rm Sing} (X)$ the collection of singularities of $X$.

\subsection{Lyapunov exponents and SRB measures}
A subset $A$ is called $\phi$-invariant if $\phi_t(A)=A$ for any $t\in \RR$. One says that a $\phi$-invariant measure $\mu$ is ergodic if any $\phi$-invariant measurable subset has $\mu$-measure one or zero.

Given a $\phi$-invariant measure $\mu$, the Oseledets theorem states that for $\mu$-almost every point $x$, there exist $k=k(x)\in \NN$, finitely many real numbers $\lambda_{1}(x)<\lambda_{2}(x)<\cdots<\lambda_{k}(x)$, and a measurable (w.r.t. $x$) invariant splitting 
$$T_{x}M=E_{1}(x)\oplus \cdots\oplus E_k(x)$$
such that for every $1\le i \le k$,
$$\lim\limits_{t\to \pm\infty} \frac{1}{t} \log \| D\phi_{t}(x)(v) \|=\lambda_{i}(x),\ \forall \ v\in E_{i}(x)\setminus\{0\}.$$ 

These numbers $\lambda_{1}(x)<\lambda_{2}(x)<\cdots<\lambda_{k}(x)$ are called the \emph{Lyapunov exponents} of $\mu$ at $x$. Note that if $\mu$ is ergodic, then the Lyapunov exponents are constant for $\mu$-almost every $x$. It can be easily showed that the Lyapunov exponents w.r.t. $\phi$ above are indeed, equal to Lyapunov exponents w.r.t. the diffeomorphism $\phi_{1}$(the time one map of $\phi$).

Let $h_{\mu}(\phi):=h_{\mu}(\phi_1)$ be the metric entropy of $\mu$ w.r.t. $\phi$.
\begin{definition}\label{SRB}
We say a $\phi$-invariant measure $\mu$ is an SRB measure if 
$$
h_{\mu}(\phi)=\int \sum\lambda^{+}(x)d\mu(x),
$$ 
where $\sum\lambda^{+}(x)$ is the sum of all the positive Lyapunov exponents counted with multiplicity.
\end{definition}

The following classical Ruelle inequality provides a connection between entropy and Lyapunov exponents.

\begin{lemma}\label{Ruelle inequality}
	\cite[Theorem 2]{Rue78} 
Let $\mu$ be a $\phi$-invariant measure, then 
$$
h_{\mu}(\phi)\le \int \sum\lambda^{+}(x)d\mu(x).
$$ 
\end{lemma}
\begin{remark}
	In particular, if $\mu$ is an ergodic measure, the integral in the right side of the inequality above can be removed (because the Lyapunov exponents of an ergodic measure are constant for almost every $x$). 	
\end{remark}

The following is a flow version of the classical ergodic decomposition theorem. It is parallel to the discrete-time version which can be found in the book \cite[Chapter II.6]{Ma}. 

\begin{proposition}\label{Ergodic decomposition of flow}
Let $\phi$ be a $C^1$ flow on a compact manifold $M$.
There exists a measurable subset $\Sigma\subset M$ with total measure ($\mu(\Sigma)=1$, for any invariant measure $\mu$) such that 
\begin{itemize}
\item for every $x\in \Sigma$, the measure
$$
\eta_x=\lim_{T\to +\infty}\frac{1}{T}\int_0^T \delta_{\phi_s(x)}ds
$$
is well defined and ergodic;
\item
every $\phi$-invariant measure has the ergodic decomposition:
$$
\mu=\int_{\Sigma} \eta_{x} d{\mu}.
$$
\end{itemize}

%
\end{proposition}

\subsection{Partially hyperbolic flows}

We say a compact $\phi$-invariant subset $\Lambda$ admits a \emph{dominated splitting} if there is a $D\phi_t$-invariant continuous splitting $T_{\Lambda}M=E\oplus F$ and two constants $C>0, \lambda>0$ such that
$$
\| D\phi_{t}|_{E_x}\| \cdot \parallel D\phi_{-t}|_{F_{\phi_{t}(x)}}\|\le  C{\rm e}^{-\lambda t}, \quad \forall\, t>0, \forall\,x\in \Lambda 
.$$
In this case, we call $F$ \emph{dominates} $E$, and denote it by $E\oplus_{\prec}F$.

Now we give the definition of partial hyperbolicity.

\begin{definition}\label{partially hyperbolic of flow}
We say a compact invariant subset $\Lambda$ admits a \emph{partially hyperbolic splitting} if there is a $D\phi_t$-invariant continuous splitting 
$$
T_{\Lambda}M=E^{ss}\oplus E^c \oplus E^{uu}
$$ such that
\begin{itemize}
\smallskip
\item $E^{uu}$ dominates $E^{c}\oplus E^{ss}$ and $E^{c}\oplus E^{uu}$ dominates $E^{ss},$
\smallskip
\item $E^{ss}$ is uniformly contracting and $E^{uu}$ is uniformly expanding, i.e., there are two constants $C>0$ and $\lambda>0$ such that for any $x\in \Lambda$ and $t>0$,
\begin{itemize}
\smallskip
\item $\|D\phi_{t}(v)\|\le C{\rm e}^{-\lambda t}\| v\|$, $\forall v \in E^{ss},$
\smallskip
\item $\|D\phi_{-t}(v)\|\le C{\rm e}^{-\lambda t}\| v\|$, $\forall v\in E^{uu}.$
\end{itemize}
\end{itemize} 

\end{definition}

\begin{remark}
	Write $E^{cs}=E^{c}\oplus E^{ss}$, $E^{cu}=E^{c}\oplus E^{uu}$. $E^{cs}$ is called the central stable direction and $E^{cu}$ is called the central unstable direction. Note by definition that if the flow has no singularity, then the flow direction (i.e., $<X>$) must be contained in $E^{c}$. 
\end{remark}

\subsection{Entropy along unstable foliation}

In this subsection, we introduce the definition of the entropy along unstable foliation.

Let $\Lambda$ be an attractor with partially hyperbolic splitting $T_{\Lambda}M=E^{cs}\oplus_{\prec} E^{uu}$. We denote the strong unstable manifold at $x\in \Lambda$ by $W^{uu}(x)$. The strong \emph{unstable lamination} on $\Lambda$ is denoted by 
$$
\mathcal{F}^{u}=\{W^{uu}(x),\ x\in \Lambda\}.
$$

\begin{definition}
Assume that $\mu$ is an invariant probability measure supported on $\Lambda$.	
A measurable partition $\alpha$ is \emph{$\mu$-subordinate to $\mathcal{F}^{u}$}, if for $\mu$-a.e. $x$, 
\begin{itemize}
	\smallskip
	\item $\alpha(x)\subset W^{uu}(x)$,
	\smallskip	
	\item $\alpha(x)$ contains an open neighborhood of $x$ inside $W^{uu}(x)$ w.r.t. the intrinsic topology on $W^{uu}(x)$.
\end{itemize} 
\end{definition}
\begin{remark}
	For the definition of measurable partition and its conditional measures, see \cite{Rok67}. The existence of subordinate partitions is showed in \cite{LeY85}. 
\end{remark}

\begin{definition}\label{GS}
	The entropy of $\mu$ along the unstable foliation $\mathcal{F}^{u}$ is defined by
	$$h_{\mu}(\phi_{1},\mathcal{F}^{u})=h_{\mu}(\phi_{1},\alpha)$$ 
	where $\alpha$ is a measurable partition $\mu$-subordinate to $\mathcal{F}^{u}$ and $\phi_{1}$ is the time one map of the flow $\phi$.
\end{definition}

\begin{remark}$ $
\begin{itemize} 
\smallskip
	\item It is showed in \cite{LeY85} that $h_{\mu}(\phi_{1},\mathcal{F}^{u})$ does not depend on the choice of subordinate partitions. Hence the definition above is well defined.
	\smallskip
	\item By definition, $h_{\mu}(\phi_{1},\mathcal{F}^{u})\leq h_{\mu}(f)$ where $h_{\mu}(f)$ is the entropy of $\mu$.
\end{itemize}
\end{remark}

\section{Existence of SRB measures} 
The proof of Theorem \ref{TheA} will be given in this section. To begin with, let us state the following two results. They establish the entropy formulas with respect to limit measures generated by the iterations of Lebesgue almost every point.

For each $x\in M$, take
$$
\mathcal{M}(x)=\Big\{\mu: \exists \ T_{n}\to \infty~\textrm{such that}~\mu=\lim\limits_{n\to \infty}\frac{1}{T_{n}} \int_{0}^{T_{n}} \delta_{\phi_{t}(x)}dt \Big\}.
$$

\begin{lemma}\label{lemme 3.2}\cite[Theorem C]{CDJ20}
Assume that $\phi$ is a flow generated by a $C^{1}$ vector field $X$. If an attractor $\Lambda$ of $\phi$ has a partially hyperbolic splitting $$T_{\Lambda}M=E^{ss}\oplus E^{c}\oplus E^{uu},$$ then for Lebesgue almost every point $x$ of $B(\Lambda)$, and for any $\mu\in \mathcal{M}(x)$, we have	
\begin{equation}\label{ue}
h_{\mu}(\phi_{1},\mathcal{F}^{u})= \int \log|\det D\phi_{1}|_{E^{uu}}|d\mu,
\end{equation}
\end{lemma}

\begin{lemma}\label{lemma 3.3}\cite[Theorem F]{CDJ20}
Assume that $\phi$ is a flow generated by a $C^{1}$ vector field $X$. If $\Lambda$ is an attractor of $\phi$ exhibiting dominated splitting $T_{\Lambda}M=E\oplus_{\prec} F$, then for Lebesgue almost every point $x$ of $B(\Lambda)$, and for any $\mu\in \mathcal{M}(x)$, we have 	
\begin{equation}\label{fe}
h_{\mu}(\phi_{1})\ge \int \log|\det D\phi_{1}|_{F}|d\mu.
\end{equation}
\end{lemma}

\medskip
Given an attractor $\Lambda\subset M$ for the flow $\phi$, use $\mathcal{M} (\phi,\Lambda)$ denotes the set of $\phi$-invariant probability measures supported on $\Lambda$.

\begin{lemma}\label{lemma 3.4}
\cite[Corollary 2.15]{CDJ20} Let $\phi$ be the flow generated by a $C^{1}$ vector field $X$, and let $\Lambda$ be an attractor of $\phi$ with partially hyperbolic splitting
$$
T_{\Lambda}M=E^{cs}\oplus E^{uu}.
$$ 
Then the set 
$$
\mathcal{M}_{u}=\Big\{{\mu \in \mathcal{M} (\phi,\Lambda)}:h_{\mu}(\phi_{1},\mathcal{F}^{u})=\int \log|\det D \phi_{1}|_{E^{uu}}|d\mu\Big\}
$$ 
is convex and compact. A measure belongs to $\mathcal{M}_{u}$ iff each of its ergodic components does.
\end{lemma}

\begin{lemma}\label{System is not a singularity}
Let $\phi$ be the flow generated by a $C^{1}$ vector field $X$. If $\Lambda$ is an attractor for $\phi$ with partially hyperbolic splitting $T_{\Lambda}M=E^{cs}\oplus_{\prec} E^{uu}$, then there is no singularity on $\Lambda$.
\end{lemma}

\begin{proof}
Going by contradiction, we assume that there is a singularity $\sigma$ on $\Lambda$, then $\sigma$ has a splitting 
$$
T_{\sigma}M=E^{cs}\oplus E^{uu}.
$$ 
Therefore, the strong unstable manifold $W^{uu}(\sigma)$ of $\sigma$ is tangent to $E^{uu}(\sigma)$. Since $\Lambda$ is an attractor, $W^{uu}(\sigma)\subset \Lambda$. Take a point $x\in W^{uu}(\sigma)\setminus\{\sigma\}$, then the invariance of $W^{uu}(\sigma)$ gives
$$
\phi_{t}(x)\in W^{uu}(\sigma)\setminus\{\sigma\},\ \ \forall t\in \mathbb{R}.
$$ 
This shows that the orbit of $x$ is in $W^{uu}(\sigma)$ (see Figure 1).
\begin{figure}[h]
	\centering
	\includegraphics[width=0.40\textwidth]{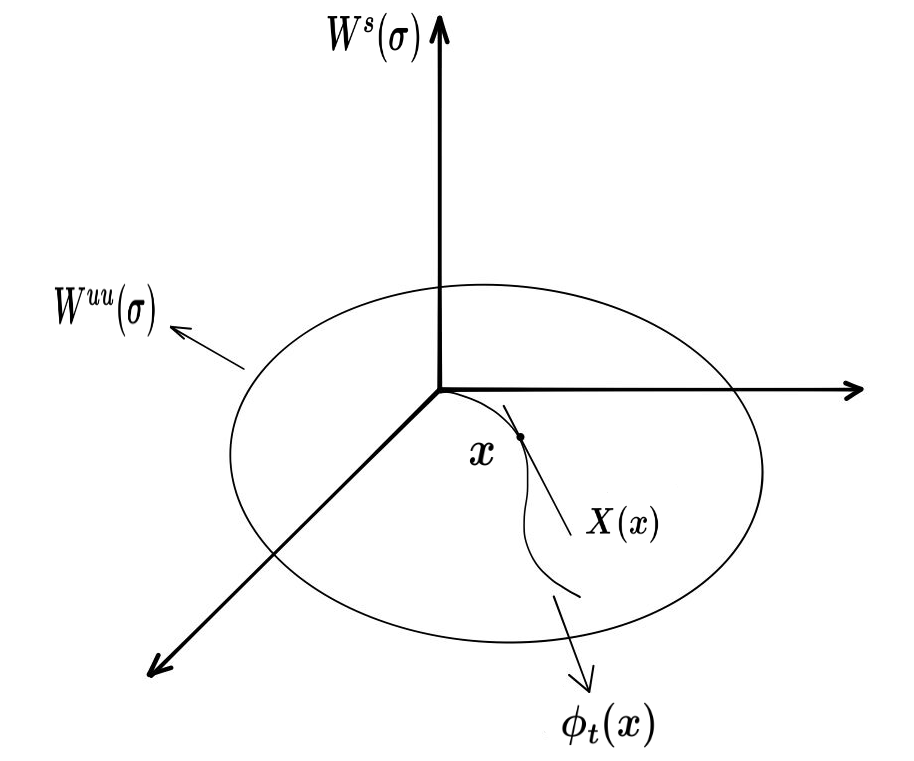}
	\caption{strong unstable manifold of singularity}
\end{figure}
In particular, 
$$
X(x)\in T_{x}W^{uu}(\sigma)= E^{uu}(x).
$$ 
Since $X$ is $C^{1}$, there exists a constant $C>0$ such that for any $x\in M$, we have $\| X(x) \| \le C$. Then \begin{equation}\label{3.4}
 \frac{\| X(\phi_{t}(x))\|}{\| X(x) \|}\le \frac{C}{\| X(x) \|}< \infty.
 \end{equation}
Since $X(\phi_{t}(x))=D\phi_{t}X(x)$ and $X(x)\in E^{uu}(x)$, there exists $\lambda>0$ such that for any $t>0$, we have 
$$
\| X(\phi_{t}(x))\| \ge {\rm e}^{\lambda t}\| X(x)\|.
$$ 
This implies that $\frac{\| X(\phi_{t}(x))\|}{\| X(x) \|}$ is unbounded as $t\to \infty$, whis is a contradiction to (\ref{3.4}). Therefore, there is no singularity on $\Lambda$.
\end{proof}

Lemma \ref{System is not a singularity} together with a recent work \cite[Theorem D]{MaS20} implies the following interesting result.

\begin{corollary}\label{TheC}
Let $X$ be a $C^{1}$ vector field on a compact manifold $M$ and let $\Lambda$ be an attractor with a partially hyperbolic splitting $T_{\Lambda}M=E^{ss}\oplus E^{c}\oplus E^{uu}$. Assume $\dim E^{c}=2$, then $\phi_{1}$ is entropy expensive\footnote{$\phi_1$ is entropy expansive if there exists $\varepsilon>0$ such that $\sup_{x\in M}{h}_{{\rm top}}(\phi_1,\ B_{\infty}(x,\varepsilon))=0$, where $B_{\infty}(x,\varepsilon)=\bigcap^{\infty}_{n=0}\phi_{-n}(B(\phi_n(x),\varepsilon)).$}. 
\end{corollary}

\begin{proof}
Assume that $\phi$ is a flow generated by $X$ with partially hyperbolic splitting $TM=E^{ss}\oplus E^{c}\oplus E^{uu}$, $\dim E^{c}=2$. According to
Rold\'{a}n-Saghin-Yang \cite[Theorem D]{MaS20}, if $\phi$ does not have singularities, then $\phi_{1}$ is entropy expansive. On the other hand, Lemma \ref{System is not a singularity} guarantees that there is no singularity on $M$. Thus we complete the proof of Corollary \ref{TheC}.
\end{proof}

\subsection{Proof of Theorem \ref{TheA}}

Theorem \ref{TheA} is a direct consequence of the following result.
\begin{theorem}
Let $\phi$ be the flow generated by $C^{1}$ vector field $X$ on $M$, assume that $\Lambda$  is an attractor of $\phi$ with a partially hyperbolic splitting 
$$
T_{\Lambda}M=E^{ss}\oplus E^{c}\oplus E^{uu} ,\ \dim E^{c}=2.
$$ 
For Lebesgue almost every point $x$ of $B(\Lambda)$ and for any $\mu\in \M(x)$, $\mu$ has an ergodic component which is an SRB measure. More precisely, we have 
\begin{itemize}
\item $\mu$ is a Gibbs $u$-state;
\item If each ergodic component of $\mu$ has only non-negative center Lyapunov exponents, then $\mu$ itself is an SRB measure;
\item If some ergodic component of $\mu$ has non-positive center Lyapunov exponents, then this ergodic component is an SRB measure.
\end{itemize}
\end{theorem}


\begin{proof}
Let $\Omega\subset B(\Lambda)$ be the full Lebesgue measure set given by Lemma \ref{lemme 3.2} and Lemma \ref{lemma 3.3}. Then for every $x\in \Omega$, each $\mu\in \mathcal{M}(x)$ satisfies the following entropy formulas
\begin{equation}\label{es0}
h_{\mu}(\phi_1,\mathcal{F}^u)=\int \log |\det D\phi_1|_{E^{uu}}|d\mu
\end{equation}
\begin{equation}\label{es1}
h_{\mu}(\phi_1)\ge \int \log|\det D\phi_{1}|_{E^{c}\oplus E^{uu}}|d\mu.
\end{equation}
If any ergodic component of $\mu$ has only non-negative center Lyapunov exponents, then 
$$
\int  \sum \lambda^{+}(z)d\mu(z)=\int \log|\det D\phi_{1}|_{E^{c}\oplus E^{uu}}|d\mu.
$$ 
Combining this with (\ref{es1}) one gets
$$
h_{\mu}(\phi_{1})\geqslant\int \log|\det D\phi_{1}|_{E^{c}\oplus E^{uu}}|d\mu=\int  \sum \lambda^{+}(z)d\mu(z).
$$ 
By Lemma \ref{Ruelle inequality}, we know
$$
h_{\mu}(\phi)=h_{\mu}(\phi_{1})\le \int \sum \lambda^{+}(z)d\mu(z).
$$
Consequently, we obtian
$$
h_{\mu}(\phi)=\int \sum \lambda^{+}(z)d\mu(z).
$$
%
Hence, $\mu$ is an SRB measure.

\medskip	
If there exists an ergodic component $\nu$ of $\mu$ which 
has non-positive center Lyapunov exponents, then by (\ref{es0}) and Lemma \ref{lemma 3.4} we know 
\begin{equation}\label{es3}
h_{\nu}(\phi_{1},\mathcal{F}^{u})=\int \log |\det D\phi_{1}|_{E^{uu}}|d\nu>0.
\end{equation}

By Poincar\'{e} recurrent theorem, $\mu$-almost every point $x$ is recurrent, i.e., there exist $t_{n}\to \infty$ such that $\phi_{t_{n}}(x)\to x$. Note that $\Lambda$ does not contain any singularities by Lemma \ref{System is not a singularity}, we have $X(x)\neq 0$. So, we get $\lim_{n\to +\infty}X(\phi_{t_{n}}(w))=X(x)\neq 0$. Observe also that $D\phi_{t_{n}}(X(x))= X(\phi_{t_{n}}(x))$. Consequently, 
\begin{eqnarray*}
\lim_{t\to +\infty}\frac{1}{t}\log \|D\phi_t(X(x))\|&=&\lim_{n\to +\infty}\frac{1}{t_n}\log \|D\phi_{t_n}(X(x))\|\\
&=&\lim_{n\to +\infty}\frac{1}{t_n}\log \|X(\phi_{t_{n}}(x))\|\\
&=&\lim_{n\to +\infty}\frac{1}{t_n}\log ||X(x)\|=0.
\end{eqnarray*}
This shows that the Lyapunov exponent of $\nu$ on $\textless X \textgreater$ is zero. Since ${\rm dim}E^c=2$, $\nu$ has at most two different Lyapunov exponents along $E^c$, thus $\nu$ has only non-positive center Lyapunov exponents.
Therefore, 
$$
\int \sum \lambda^{+}(z)d\nu(z)=\int \log|\det D\phi_{1}|_{E^{uu}}|d\nu.
$$ 
This together with Lemma \ref{Ruelle inequality} and (\ref{es3}) ensures that
$$
h_{\nu}(\phi_{1})\le \int \sum \lambda^{+}(z)d\nu(z)=\int \log|\det D\phi_{1}|_{E^{uu}}|d\nu= h_{\nu}(\phi_{1},\mathcal{F}^u)\le h_{\nu}(\phi_1).
$$
Thus, we obtain
$$
h_{\nu}(\phi_{1})= \int \sum \lambda^{+}(z)d\nu(z),
$$
which shows that $\nu$ is an SRB measure.

\end{proof}

\section{The existence and finiteness of physical measures}
In this section, we will give the proof of Theorem \ref{TheB} and Corollary \ref{uniqueness}. In $\S$ \ref{e}, we give the existence of physical measures. In $\S$ \ref{f}, we establish the finiteness and uniqueness of physical measures
by using the tool of linear Poincar\'{e} flow.

\subsection{The existence of physical measures}\label{e}
%
%
%

According to Ruelle inequality and ergodic decomposition theorem, we get the following result. 
%

\begin{lemma}\label{lemma SRB}\cite[Lemma 3.3]{HCMP}
If $\mu$ is an SRB measure, then for $\mu$-almost every $x$, $\eta_x$ is an SRB measure. 
\end{lemma}

\begin{lemma}\label{nz}
Let $\phi$ be the flow generated by a $C^{1}$ vector field $X$, and let $\Lambda$ be an attractor for $\phi$ with partially hyperbolic splitting 
$$
T_{\Lambda}M=E^{ss}\oplus E^{c}\oplus E^{uu}.
$$ 
If $E^c$ is Gibbs sectional expanding, then for any Gibbs $u$-state $\mu$, all its Lyapunov exponents along $E^c$ are positive except along the flow direction.
\end{lemma}
\begin{proof}
Let $\mu$ be a Gibbs $u$-state of $\phi$. For $\mu$-almost every point $x$, let $v$ be a non-zero vector in $E^c\setminus <X>$. We consider the plane $L_x$ spanned by $v$ and $X(x)$. Since the Lyapunov exponent along $X(x)$ is zero, the Gibbs sectional expansion condition implies that
$$
\lim_{t\to +\infty}\frac{1}{t}\log\|D\phi_t(v)\|=\lim_{t\to +\infty}\frac{1}{t}\log \left|{\rm det}D\phi_t|_{L_x}\right|>0.
$$
This completes the proof.
\end{proof}

Now we show the existence of physical measures for flows with Gibbs sectional expanding center. 

\begin{theorem}\label{Gibbs expanding exist of physical measure}
Let $\phi$ be the flow generated by a $C^{2}$ vector field $X$, and let $\Lambda$ be an attractor with partially hyperbolic splitting 
$$
T_{\Lambda}M=E^{ss}\oplus E^{c}\oplus E^{uu}.
$$ 
If $E^{c}$ is Gibbs sectional expanding, then there exist physical measures supported on $\Lambda$. Moreover, there exists a subset $\Omega\subset B(\Lambda)$ of full Lebesgue measure such that for any $x\in \Omega$, every ergodic component of each measure of $\M(x)$ is a physical measure supported on $\Lambda$.
\end{theorem}

\begin{proof}
By Lemma \ref{lemme 3.2} and Lemma \ref{lemma 3.3}, there is $\Omega\subset B(\Lambda)$ with full Lebesgue measure such that for any $x\in \Omega$, each $\mu\in \M(x)$ is a Gibbs $u$-state and satisfies 
$$
h_{\mu}(\phi_1)\ge\int |{\rm det} D\phi_1|_{E^c\oplus E^{uu}}|d\mu.
$$
Since $E^c$ is Gibbs sectional expanding, Lemma \ref{nz} implies that the Lyapunov exponents of $\mu$ along $E^c$ are non-negative. Thus, Lemma \ref{Ruelle inequality} gives 
$$
\int |{\rm det} D\phi_1|_{E^c\oplus E^{uu}}|d\mu=\int \sum\lambda^{+}d\mu.
$$
Therefore, 
$$
h_{\mu}(\phi_1)=\int |{\rm det} D\phi_1|_{E^c\oplus E^{uu}}|d\mu=\int \sum\lambda^{+}d\mu.
$$
Moreover, by applying Lemma \ref{nz}, Lemma \ref{lemma 3.4} and Oseledets Theorem,  for every ergodic component $\nu$ of $\mu$, there exists the finer invariant splitting $E^c=E_{\nu}^{c}\oplus <X>$ on a full $\nu$-measurable set such that $\nu$ has positive Lyapunov exponents along $E_{\nu}^{c}$. 
It follows from Lemma \ref{lemma SRB} that
$$
h_{\nu}(\phi_1)=\int |{\rm det} D\phi_1|_{E_{\nu}^{c}\oplus E^{uu}}|d\nu=\int \sum\lambda^{+}(x)d\nu(x).
$$
Thus, the conditional measures along local Pesin unstable manifolds tangent to $E_{\nu}^{c}\oplus E^{uu}$ are equivalent to the Lebesgue measures (see e.g. \cite[Theorem 13.1.2]{BP07}). Since $\nu(B(\nu))=1$, there exists a local unstable manifold $W^u_{loc}(x)$ tangent to $E_{\nu}^{c}\oplus E^{uu}$ such that Lebesgue almost every point of $W^u_{loc}(x)$ belongs to $B(\nu)$.
Take $s>0$ and put
$$
D_s=\bigcup_{y\in \phi_{[-s,s]}W^u_{loc}(x)}W^{ss}(y), 
$$
where $\phi_{[-s,s]}W^u_{loc}(x)=\left\{\phi_t(z): z\in W^u_{loc}(x), t\in [-s,s]\right\}$ and $W^{ss}(y)$ denotes the strong stable manifold of $y$ (tangent to $E^{ss}$).
Observe that $B(\nu)$ is $s$-saturated, that is, if $z\in B(\nu)$, then $W^{ss}(z)\subset B(\nu)$. By absolute continuity of stable lamination (see e.g. \cite[Theorem 7.1]{P04}), we know that Lebesgue almost every point of $D_s$ belongs to $B(\nu)$. Thus,
$$
{\rm Leb}(B(\nu))\ge{\rm Leb}(D_s\cap B(\nu))={\rm Leb}(D_s)>0.
$$ 
This shows that $\nu$ is a physical measure.
\end{proof}

\subsection{The finiteness and uniqueness of physical measures}\label{f}
\subsubsection{Linear Poincar\'{e} flow}

In order to prove the finiteness of physical measures for systems considered in Theorem \ref{TheB}, one need to find local unstable manifolds with uniform size for typical points . In the case of vector field, the main difficulty is to control the angle between the subspace with positive Lyapunov exponents and the direction of vector field. In order to conquer this, we next introduce the linear Poincar\'{e} flow. It is an induced flow of the original flow $\phi$ that the dynamics along flow direction somehow is ignored.

\medskip
For any $x\in M\setminus {\rm Sing}(X)$, we define the normal space of $X(x)$ by
$$
\N_{x}=\left\{v\in T_{x}M:\ \langle v, X(x)\rangle=0 \right\},
$$ 
where $\langle\cdot,\cdot\rangle$ is inner product on $T_{x}M$ given by Riemannian metric. Denote by $\pi_x : T_xM \mapsto \N_x$ the orthogonal projection of $T_xM$ onto $\N_x$. 

\medskip
\paragraph{\bf{Linear Poincar\'{e} flow}}
For any $x\in  M\setminus  {\rm Sing}(X)$, $t\in \mathbb{R}$, the \emph{linear Poincar\'{e} flow}
$\psi_t: \N_x\to \N_{\phi_t(x)}$ is defined by $\psi_t(v)=\pi_{\phi_t(x)}D\phi_t(v)$ for every $v\in \N_x$. More precisely,
$$
\psi_{t}(v)=D\phi_{t}(v)-\frac{\left\langle D\phi_{t}(v),X(\phi_{t}(x))\right\rangle}{\parallel X(\phi_{t}(x))\parallel^{2}}X(\phi_{t}(x)), \quad \forall v\in \mathscr{N}_x.
$$

\begin{figure}[h]
	\centering
	\includegraphics[width=0.60\textwidth]{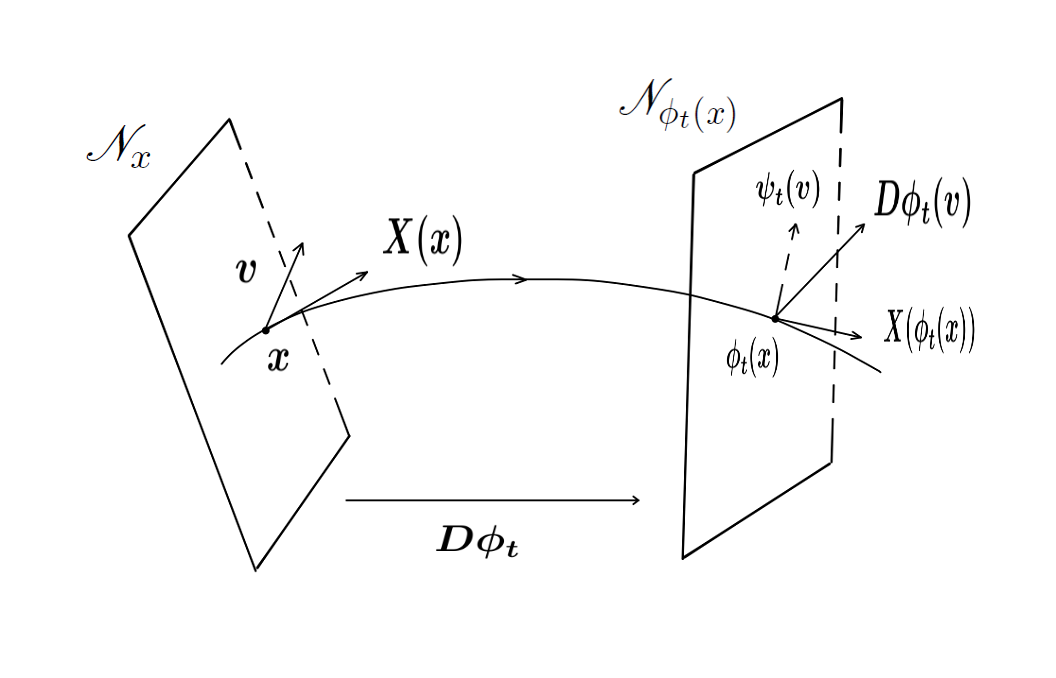}
	\caption{linear Poincar\'{e} flow}
\end{figure}

\noindent
\paragraph{\bf{Dominated splitting for linear Poincar\'{e} flow}} Let $\Lambda\subset M\setminus {\rm Sing}(X)$ be a $\phi$-invariant set, we say that the linear Poincar\'{e} flow admits a dominated splitting on $\Lambda$, if there exists a $\psi_{t}$-invariant splitting 
$$
\N_{\Lambda}=\mathscr{E}\oplus \mathscr{F}
$$ 
and constants $C>0$, $\eta>0$ such that
$$
\| \psi_{t}|_{\E(x)}\|\cdot\| \psi_{-t}|_{\F({\phi_{t}(x)})}\|\le C{\rm e}^{-\eta t}\quad \textrm{for every}~x\in \Lambda ~\textrm{and}~t>0. 
$$

\medskip
\paragraph{\bf{Partially hyperbolic splitting for linear Poincar\'{e} flow}} Given a $\phi$-invariant set $\Lambda\subset M\setminus {\rm Sing}(X)$, then the $\psi_t$-invariant splitting 
$$
\N_{\Lambda}=\N^{s}\oplus \N^{c}\oplus \N^{u}
$$ 
is partially hyperbolic for linear Poincar\'{e} flow, if there exist constants $C>0$, $\eta>0$ such that for every $x\in \Lambda$ and $t>0$,  
\begin{itemize}
\smallskip
\item
$$
\big\| \psi_{t}|_{\N^s_{x}}\big\|\cdot \big\| \psi_{-t}|_{\N^c_{\phi_t(x)}\oplus \N^u_{\phi_t(x)}}\big\|\le C{\rm e}^{-\eta t},
$$
$$ 
\big\| \psi_{t}|_{\N^s_x\oplus \N^c_x}\big\|\cdot\big\| \psi_{-t}|_{\N^u_{\phi_{t}(x)}}\big\|\le C{\rm e}^{-\eta t};
$$
\item $\| \psi_{t}(v)\|\le C{\rm e}^{-\eta t}\| v\|, \forall v\in  \N_x^{s}$;
\smallskip
\item $\| \psi_{-t}(v)\| \le C{\rm e}^{-\eta t}\|v\|,\  \forall v\in \N_x^{u}.$
\end{itemize}


\medskip
Assume that $\Lambda\subset M\setminus {\rm Sing}(X)$ is a $\phi$-invariant set with a partially hyperbolic splitting 
$$
T_{\Lambda}M=E^{ss}\oplus E^{c} \oplus E^{uu}.
$$
By taking $\N_x^s=\pi_x E_x^{ss}$, $\N_x^c=\pi_x E_x^{c}$, $\N_x^u=\pi_x E_x^{u}$ for every $x\in \Lambda$, a standard argument (see e.g. \cite[Theorem 2.27]{AP}) shows that 
$$
\N_{\Lambda}=\N^{s}\oplus \N^{c}\oplus \N^{u}
$$
is a partially hyperbolic splitting on $\Lambda$ for linear Poincar\'{e} flow. Noting that
\begin{itemize}
\smallskip
\item[--] $\N^{s}$(resp. $\N^u$) is uniformly contracting (resp. expanding) w.r.t. $\psi_t$;
\smallskip
\item[--] ${\rm dim} \N^{c}=\dim E^{c}-1$.
\end{itemize}

\subsubsection{Uniform estimate on unstable manifolds}
For $x\in M\setminus {\rm Sing}(X)$, for any $\beta>0$ small enough, we consider the following normal manifold of $x$:
$$
\Sigma_{x}(\beta)={\rm exp}_{x}\N_{x}(\beta),
$$
where  
$
\N_x(\beta)=\{v\in \N_x: \| v\|<\beta \}.
$

Note that for any $T>0$, there exists $\beta>0$ such that the flow $\phi$ defines a holonomy map $P_{x,T}$ from $\Sigma_{x}(\beta)$ to $\Sigma_{\phi_{T}(x)}(\beta)$. Sometimes it is abbreviated to $P_{T}$. Similarly, we can define $P_{-T}$.

The proof of the following result can be found in \cite[Lemma 2.18]{GaY18}.

\begin{lemma}\label{dim estimate}
Let $\Lambda$ be a compact invariant set of flow $\phi$ without singularities, assume that $\Lambda$ admits a partially hyperbolic splitting $\N_{\Lambda}=\N^{s}\oplus \N^{c}\oplus \N^{u}$ for linear Poincar\'{e} flow. Then for any $\eta>0,\ T>0$, there exist $\delta>0$, $\beta>0$ such that for any $x\in \Lambda$, if
$$
\prod_{i=0}^{n-1}\left\|\psi_{-T}|_{\N^{c}_{\phi_{-iT}(x)}}\right\|\le {\rm e}^{-\eta n},
$$ 
then there exists a disk $W^{c}(x)$ centered at $x$ on $\Sigma_x(\beta)$ with following properties:
\begin{itemize}
\smallskip
\item $W^{c}(x)$ is tangent to $\N^{c}\oplus \N^{u}$ with dimension ${\rm dim} (\N^{c}\oplus \N^{u})$;
\smallskip
\item the size of $W^{c}(x)$ is \ $\delta$;
\smallskip
\item $\lim_{t\to +\infty}{\rm diam}(P_{-t}(W^{c}(x))=0$.
\end{itemize} 
In particular, when $x$ is a periodic point, then the unstable manifold of ${\rm Orb}(x)$ contains a sub-manifold of size $\delta$ with dimension $\dim M-\dim \N^{s}$.
\end{lemma} 

\subsubsection{Periodic approximation}
 
Let $\gamma$ be a periodic orbit, $\pi(\gamma)$ is the period of $\gamma$. For any $x\in \gamma$, let 
$$
\mu_{\gamma}=\frac{1}{\pi(\gamma)}\int_{0}^{\pi(\gamma)}\delta_{\phi_{t}(x)}dt
$$ 
be the \emph{periodic measure} of $\gamma$.

We say that a $\phi$-invariant measure $\mu$ is a \emph{hyperbolic} if it is not supported on singularities and the Lyapunov exponents of $\mu$ is nonzero, except along the direction of the flow.


\begin{definition}
Let $\gamma$ be a hyperbolic periodic orbit, and let $\mu$ be a hyperbolic ergodic measure of $C^{2}$ flow $\phi$, if for $\mu$-almost every $x$,
$$
W^{s}(x)\pitchfork W^{u}(\gamma)\ne \emptyset,\ \ W^{u}(x)\pitchfork W^{s}(\gamma)\ne \emptyset,
$$
then we say that $\gamma$ is related to $\mu$ 
\end{definition}

Let $\{\varphi_{n}\}_{n=0}^{\infty}$ be a countable dense subset in $C^{0}(M,\mathbb{R})$, we can define a distance between two probability measures $\mu,\nu$ by 
$$
\mathrm{d}(\mu,\nu)=\sum_{n=0}^{\infty}\frac{|\int \varphi_{n}d\mu- \varphi_{n}d\nu|}{2^{n}\cdot \sup_{x\in M} |\varphi_{n}(x)|}.
$$ 

Katok \cite{Kat95} proves that any hyperbolic measure can be approximated by a periodic measure for $C^{2}$ diffeomorphisms. For the case of flow, see \cite{LiY12}.

\begin{lemma}\label{Katok shadowing}
Let $\phi$ be the flow generated by a $C^{2}$ vector field $X$. Assume that $\mu$ is an ergodic measure which is not supported on singularities. If $\mu$ is a hyperbolic ergodic measure, then for any $\varepsilon>0$, there exists a hyperbolic periodic orbit $\gamma$ such that 
\begin{itemize}
\smallskip
\item $\gamma$ is related to $\mu$;
\smallskip
\item ${\rm d}(\mu, \mu_{\gamma})<\varepsilon$.
\end{itemize}
\end{lemma}

\subsubsection{Complete the proof of Theorem \ref{TheB}}
In the rest of this section, $\Lambda$ will be assumed as an attractor with partially hyperbolic splitting 
$$
T_{\Lambda}M=E^{uu}\oplus E^{c}\oplus E^{ss}.
$$ 
By Lemma \ref{System is not a singularity}, there is no singularity on $\Lambda$. Thus, one has the corresponding partially hyperbolic splitting for linear Poincar\'{e} flow:
$$
\N_{\Lambda}=\N^{u}\oplus \N^{c}\oplus \N^{s}.
$$

The following result asserts that Gibbs sectional expanding implies the uniform estimate of linear Poincar\'{e} flow on center direction.

\begin{proposition}\label{Gibbs estimate}
Let $\phi$ be a flow generated by a $C^{2}$ vector field $X$. Assume that $\Lambda$ is an attractor exhibiting a partially hyperbolic splitting 
$$
T_{\Lambda}M=E^{uu}\oplus E^{c}\oplus E^{ss}.
$$ 
If $E^{c}$ is Gibbs sectional expanding, then there exist $\eta>0, T\in \NN$ such that 
$$
\int \log \parallel \psi_{-T}|_{\N^{c}}\parallel d\mu<-\eta
$$ 
for any Gibbs $u$-state $\mu$ supported on $\Lambda$.	
\end{proposition}

\begin{proof}
By Lemma \ref{nz}, for any Gibbs $u$-state $\mu$ supported on $\Lambda$, for $\mu$-almost every point $x$, we have
\begin{equation}\label{se}
\lim_{t\to \infty}\frac{1}{t}\log \| D\phi_{-t}(v)\|<0,~\forall v\in E^{c}(x)/ <X>.
\end{equation}
Note that if $v\in N^{c}(x)$, then $\left< v,<X>\right>=0$. By definition of linear Poincar\'{e} flow, one gets $\| \psi_{-t}(v)\|\le \| D\phi_{-t}(v)\|.$ Thus, (\ref{se}) yields that for $\mu$-almost every $x\in \Lambda$, we have
$$
\lim_{t\to \infty}\frac{1}{t}\log \| \psi_{-t}|_{\N^c_x}\|<0.
$$
Consequently, for any Gibbs $u$-state $\mu$, we have 
$$
\lim_{t\to \infty}\frac{1}{t}\int \log \| \psi_{-t}|_{\N^{c}}\| d\mu<0.
$$
So there is $T_{\mu}\in \NN$ and $\alpha_{\mu}>0$ such that 
$$
\frac{1}{T_{\mu}}\int \log \| \psi_{-T_{\mu}}|_{\N^{c}} \| d\mu<-\alpha_{\mu}.
$$ 
Observe that $x\mapsto \log \| \psi_{-T_{\mu}}|_{\N^{c}(x)} \|$ is continuous, there exists a neighborhood $U_{\mu}$ of $\mu$ such that 
$$
\frac{1}{T_{\mu}}\int \log \| \psi_{-T_{\mu}}|_{\N^{c}} \| d\nu<-\alpha_{\mu},\ \forall\ \nu\in U_{\mu}.
$$ 
Recall that $\mathcal{M}_{u}$ is the set of all Gibbs $u$-state. Since $\mathcal{M}_{u}$ is compact guaranteed by Lemma \ref{lemma 3.4}, one can choose finitely many Gibbs $u$-states $\mu_{1},\cdots,\mu_{\ell}$ such that their corresponding neighborhoods can cover $\mathcal{M}_{u}$. Let 
$$
T_{i}=T_{\mu_{i}}, \quad U_{i}=U_{\mu_{i}},
\quad \alpha_{i}=\alpha_{\mu_{i}}, \ 1\le i\le \ell.
$$ 
Take $T=\prod_{i=1}^{\ell}T_{i}$, $\eta=\min\left\{T\cdot{\alpha_{i}}:\ 1\le i\le \ell \right\}$. By construction, for any Gibbs $u$-state $\mu$, there exists a neighborhood $U_{i}$ such that $\mu\in U_{i}$. For each $x\in \Lambda$, we have 
$$
\left\| \psi_{-T}|_{\N^{c}(x)}\right\| \le \prod_{k=0}^{T/T_{i}-1} \left\| \psi_{-T_{i}}|_{\N^{c}(\phi_{-kT_{i}}(x))} \right\|.
$$ 
Combined with the invariance of $\mu$, one gets
\begin{eqnarray*}
\int \log \| \psi_{-T}|_{\N^{c}(x)} \| d\mu &\le&  \sum_{k=0}^{T/T_i-1}\int \log \| \psi_{-T_{i}}|_{\N^{c}(\phi_{-kT_{i}}(x))} \| d\mu \\ 
&=& \frac{T}{T_{i}}\int \log \| \psi_{-T_{i}}|_{\N^{c}(x)} \|d\mu \\
&<& -\frac{T}{T_{i}}\alpha_{i}T_{i} \\
&=& -T\cdot\alpha_{i} \\
&\le& -\eta.
\end{eqnarray*}
By the arbitrariness of $\mu$, we get the result.
\end{proof}

The following version of Pliss lemma helps us to select good points with sufficient hyperbolic property. See \cite[Lemma 2.1]{MCY18} for a proof.

\begin{lemma}\label{Pliss lemma}
Given $C\in \mathbb{R}$, if $\{a_{n}\}$ is a sequence of real numbers satisfying 
$$
\limsup_{n\to \infty}\frac{1}{n}\sum_{i=1}^{n}a_{i}<C,
$$ 
then there exists $k\in \NN$ such that 
$$
\frac{1}{n}\sum_{i=1}^{n}a_{i+k}<C,\ \forall n\in N.
$$
\end{lemma}

The following result provides a periodic approximation of ergodic physical measures.
\begin{lemma}\label{Th}
Let $\phi$ be a flow generated by a $C^{2}$ vector field $X$. Assume that $\Lambda$ is an attractor exhibiting a partially hyperbolic splitting 
$$
T_{\Lambda}M=E^{uu}\oplus E^{c}\oplus E^{ss}.
$$ 
If $E^{c}$ is Gibbs sectional expanding, then there exist $\eta>0, T>0$ such that for any ergodic physical measure $\mu$, there exists some hyperbolic periodic orbit $\gamma$ such that
\begin{itemize}
\smallskip
\item $\gamma$ is related to $\mu$;
\smallskip
\item $\mu_{\gamma}$ is ergodic for $\phi_T$;
\smallskip
\item 
$
\int \log \parallel \psi_{-T}|_{\N^{c}}\parallel d\mu_{\gamma}<-\eta.
$ 
\end{itemize}	
\end{lemma}

\begin{proof}
Since $E^{c}$ is Gibbs sectional expanding, by Proposition \ref{Gibbs estimate}, there exist $T_0\in \NN$, $\eta_0>0$ such that 
\begin{equation}\label{ee}
\int \log \| \psi_{-T_0}|_{\N^{c}}\| d\mu<-\eta_0
\end{equation}
for any Gibbs $u$-state $\mu$. 

Choose $\varepsilon>0$ so that $\eta:=\eta_0-\varepsilon>0$. Since $\psi_t$ is close to identity as long as $t$ is close to zero, there exists $\delta>0$ such that 
\begin{equation}\label{ftl}
\log \|\psi_t|_{\N^c(x)}\|<\varepsilon,\quad \forall |t|<\delta, ~\forall x\in \Lambda.
\end{equation}

Let $\P=\{\mu_{\alpha}: \alpha\in I\}$ be the family of ergodic physical measures.  We know $I$ is countable by the definition of physical measure. 

Take any $\mu_{\alpha}\in \P$, then it is a Gibbs $u$-state\cite[Chapter 11]{BoD05}, and thus it satisfies (\ref{ee}). 
It follows from Lemma \ref{nz} that $\mu_{\alpha}$ is hyperbolic. In view of (\ref{ee}) and Lemma \ref{Katok shadowing}, one can take a hyperbolic periodic orbit $\gamma_{\alpha}$ such that 
\begin{itemize}
\smallskip
\item $\gamma_{\alpha}$ is related to $\mu_{\alpha}$;
\smallskip
\item $\mu_{\gamma_{\alpha}}$ is close enough to $\mu_{\alpha}$ so that 
\begin{equation}\label{eg}
\int\log\| \psi_{-T_0}|_{\N^{c}}\|d\mu_{\gamma_{\alpha}}<-\eta_0.
\end{equation}
\end{itemize} 
By the result of \cite{PS}, for every $\alpha\in I$, there exists a countable subset $\T_{\alpha}$ of $\RR$ such that $\mu_{\gamma_{\alpha}}$ is ergodic for $\phi_t$ whenever $t\in \RR\setminus \T_{\alpha}$. Observe that $\cup_{\alpha\in I}\T_{\alpha}$ is countable, we can fix any
$$T\in [T_0, T_0+\delta)\setminus \bigcup_{\alpha\in I}\T_{\alpha}.$$
As a consequence, $\mu_{\gamma_{\alpha}}$ is ergodic for $\phi_T$ for each $\alpha\in I$. Moreover, estimates (\ref{ftl}), (\ref{eg}) yield
\begin{eqnarray*}
\int\log\| \psi_{-T}|_{\N^{c}}\|d\mu_{\gamma_{\alpha}}&\le& \int\log\| \psi_{-T_0}|_{\N^{c}}\|d\mu_{\gamma_{\alpha}}+\int\log\| \psi_{-(T-T_0)}|_{\N^{c}}\|d\mu_{\gamma_{\alpha}}\\
&<&-\eta_0+\varepsilon\\
&=& -\eta<0.
\end{eqnarray*}
This completes the proof.
\end{proof}

Now we can show the finiteness of physical measures under the setting of Theorem \ref{TheB}.

\begin{theorem}\label{Gibbs expanding finite physical measure}
Let $\phi$ be the flow generated by a $C^2$ vector field $X$. Assume that $\Lambda$ is an attractor with partially splitting 
$$
T_{\Lambda}M=E^{uu}\oplus E^{c}\oplus E^{ss}.
$$ 
If $E^{c}$ is Gibbs sectional expanding, then there are finitely many SRB measures (which are physical) $\mu_{1},\mu_{2},\cdots,\mu_{k}$ supported on $\Lambda$ such that 
$$
\sum_{i=1}^{k}{\rm Leb}(B(\mu_{i}))={\rm Leb}(B(\Lambda)).
$$
\end{theorem}

\begin{proof}
%
Let $\eta>0$ and $T>0$ be the constants given by Lemma \ref{Th}. Let $\delta>0$ be the constant related to $T$ and $\eta$ given by Lemma \ref{dim estimate}.

For any ergodic physical measure $\nu$, by applying Lemma \ref{dim estimate}, one gets that
there exists a hyperbolic periodic orbit $\gamma$ such that 
\begin{itemize}
\smallskip
\item $\gamma$ is related to $\nu$;
\smallskip
\item $\mu_{\gamma}$ is ergodic for $\phi_T$ and satisfies 
$$
\int\log\| \psi_{-T}|_{\N^{c}}\|d\mu_{\gamma}<-\eta.
$$
\end{itemize} 
Therefore, the Birkhoff ergodic theorem implies that there is $x\in \gamma$ such that
$$
\lim_{n\to \infty}\frac{1}{n}\sum_{i=0}^{n-1}\log\| \psi_{-T}|_{\N^{c}_{\phi_{-iT}(x)}}\|<-\eta.
$$ 
Let 
$$
a_{i+1}=\log\| \psi_{-T}|_{\N^{c}_{\phi_{-iT}(x)}}\|,\ i\ge 0.
$$ 
By applying Lemma \ref{Pliss lemma}, there exists $k\in N$ such that $y=\psi_{-kT}(x)$ satisfies 
$$
\prod_{i=0}^{n-1}\left\| \psi_{-T}|_{\N^{c}_{\phi_{-iT}(y)}}\right\|<{\rm e}^{-\eta n},\ \forall n\in \NN.
$$
It follows from Lemma \ref{dim estimate} that there is a disk $W^{c}(y)$ at $y$ with dimension $\dim (\N^{c}\oplus \N^{u})$ and size $\delta$. Moreover, ${\rm diam}\left(P_{-t}(W^{c}(y))\right)\to 0$ as $t\to \infty.$ 
Because $\nu$ is an ergodic SRB measure which is s related to $\gamma$, there exists a Pesin unstable manifold $W^{u}_{loc}(z)$ such that 
\begin{enumerate}
\smallskip
\item\label{11} Lebesgue almost every point of $W^{u}_{loc}(z)$ belongs to $B(\nu)$;
\smallskip
\item\label{22} $W_{loc}^{u}(z)\pitchfork W^{s}(\gamma)\ne \emptyset$;
\smallskip
\item\label{333} $W_{loc}^{u}(z)\subset {\rm supp}(\nu)$.
\end{enumerate}
Moreover, from Item (\ref{11}), up to considering the iteration of $z$ under the flow, we may assume that $W_{loc}^{u}(z)\pitchfork W^{ss}(y)\ne \emptyset$. By applying $\lambda$-Lemma(see e.g. \cite[Theorem 5.7.2]{BG02}), there exists $T_0>0$ and a sub-disk $D_0$ of $W_{loc}^{u}(z)$ so that $D_{\delta}:=\phi_{T_0}(D_0)$ is $C^1$-close to $W^{c}(y)$ with radius $\delta$. Following Item (\ref{22}) and the invariance of $B(\nu)$, Lebesgue almost every point of $D_{\delta}$ belongs to $B(\nu)$.
Consider the set
$$
\widehat{D}_{\nu,\delta}=\bigcup_{x\in \phi_{[0,\delta]}D_{\delta}}W^{ss}(x),
$$
where $\phi_{[-\delta,\delta]}D_{\delta}=\left\{\phi_t(z): z\in D_{\delta}, t\in [-\delta,\delta]\right\}$. The absolute continuity of stable foliation implies that $B(\nu)$ has full Lebesgue measure in $\widehat{D}_{\nu,\delta}$. Since the size of $\widehat{D}_{\nu,\delta}$ is bounded below by a constant independent of $\nu$, there is $\zeta>0$ independent of $\nu$ so that 
$$
{\rm Leb}(B(\nu))\ge\zeta.
$$
As a consequence, we know there are finitely many ergodic physical measures. 
%
%

By Theorem \ref{Gibbs expanding exist of physical measure}, there exists a subset $\Omega\subset B(\Lambda)$ of full Lebesgue measure such that for any $x\in \Omega$, for any $\mu\in \M(x)$, every ergodic component of $\mu$ is a physical measure. By the above argument, the number of these ergodic physical measures is finite. Denote these physical measures by
$\mathscr{A}=\{\mu_1,\cdots,\mu_k\}$.
To show the basin covering property, it suffices to check the following claim: 
\begin{claim}\label{cl}
For any $x\in \Omega$, there exists $i\in \{1,2,\cdots,k\}$ such that $\mathcal{M}(x)=\{\mu_{i}\}$.
\end{claim}
\end{proof}
\begin{proof}[Proof of Claim \ref{cl}]
According to the above argument, for each $\mu_i\in \mathscr{A}$, one can find an open subset $U_i:={\rm int}\widehat{D}_{\mu_i,\delta}$ such that
\begin{enumerate}[(I)]
\smallskip
\item\label{sed} ${\rm Leb}(U_i)={\rm Leb}(U_i\cap B(\mu_{i}))$,
\smallskip
\item\label{sg} $U_i\cap {\rm supp}(\mu_{i})\ne \emptyset$ (recalling Item (\ref{333})).
\end{enumerate}

Take any $\mu\in \mathcal{M}(x)$. By Proposition \ref{Ergodic decomposition of flow}, there exist $\alpha_{i}\in [0, 1],\ 1\le i\le k$ such that 
$$
\mu=\sum_{i=1}^{k}\alpha_{i}\mu_{i},\quad \sum_{i=1}^{k}\alpha_{i}=1.
$$ 
We first prove that there exists some $i\in \{1,2,\cdots,k\}$ such that $\mu=\mu_{i}$. By contradiction, there exists $i,\ j\in \{1,2,\cdots,k\},\ i\ne j$ such that $\alpha_{i}>0, \alpha_{j}>0$. 
By (\ref{sg}), one can choose $y\in {\rm supp}(\mu_{i})\cap U_i$ and $z\in {\rm supp}(\mu_{j})\cap U_j$, then $y,z\in {\rm supp}(\mu)$. By then construction of $\mu$, $\phi_t(x)$ visits any fixed neighborhoods of $y,z$ for infinitely many times.
As a result, there is $t_{0}>0$ such that 
$$
\phi_{t_{0}}(U_{i})\cap U_{j}\ne \emptyset.
$$ 
Then the invariance of basin and (\ref{sed}) imply $B(\mu_{i})\cap B(\mu_{j})=\emptyset$, so $\mu_i=\mu_j$, which contradicts to our assumption. 

\medskip
Now we suppose $\mu_{l},\ \mu_{m}\in \mathcal{M}(x)$, following (\ref{sed}), with the same argument one can deduce that 
there exists $t_{1}>0$ such that 
$$
\phi_{t_{1}}(U_{l})\cap U_{m}\ne \emptyset.
$$
Applying (\ref{sg}) again, we get $\mu_{l}=\mu_{m}$, i.e., $l=m$. This completes the proof.

\end{proof}

Now we prove the uniqueness of SRB(physical) measures under the assumption of transitivity.

\begin{proof}[Proof of Corollary \ref{uniqueness}]
By Theorem \ref{TheB}, there exist finitely many SRB(physical) measures $\mu_1,\cdots,\mu_k$ supported on $\Lambda$, whose union of basins cover a full Lebesgue measure. Moreover, for each $\mu_i$, $1\le i \le k$, there is an open neighborhood $U_i$ containing points from $\Lambda$, it satisfies that ${\rm Leb}(U_i)={\rm Leb}(U_i\cap B(\mu_i))$. Then the transitivity of $\phi$, and invariance of basins implies that there exists only one SRB(physical) measure supported on $\Lambda$ with basin covering property.
\end{proof}

\end{document}